\documentclass[12pt]{amsart} 
\usepackage{amsmath, amssymb, amsthm, bm, mathtools, enumitem, caption}
\usepackage{hyperref}
\usepackage{enumitem}
\usepackage{float}
\usepackage{listings}
\usepackage{xcolor}
\definecolor{codegray}{gray}{0.95}
\usepackage[noabbrev,capitalize]{cleveref}
\usepackage{adjustbox}
\crefname{equation}{}{}
\usepackage{ytableau}
\usepackage{graphicx, tikz}
\usepackage[textheight=8.80in, textwidth=6.85in]{geometry}

\usepackage{microtype}
\usepackage{bbm}

\newtheorem{theorem}{Theorem}
\newtheorem{lemma}[theorem]{Lemma}
\newtheorem{corollary}[theorem]{Corollary}

\newtheorem*{conjecture*}{Conjecture}

\theoremstyle{definition}

\theoremstyle{remark}
\newtheorem*{remark}{Remark}

\newtheorem*{example}{Example}

\numberwithin{equation}{section}

\usepackage{bm} 

\DeclareMathOperator{\lcm}{lcm}

\title[Partition functions that repel perfect-powers]{Partition functions that repel perfect-powers}


\author{Ken Ono}

\address{Dept. of Mathematics, University of Virginia, Charlottesville, VA 22904, USA}
\email{ko5wk@virginia.edu}

\begin{document}
\maketitle

\begin{abstract}
A conjecture by Sun states that the partition function $p(n)$, for $n>1$, is never a perfect power. Recent work by Merca et al. proposes generalizations of perfect-power repulsion for $p(n)$.  In this note, we prove these generalizations for the functions $p_B(n)$, which count the number of partitions of $n$ with the largest part $\leq B$. If $B\geq 4$ and $k\geq 3$, with $k\nmid (B-1)$, then we prove that there are only finitely many pairs $(n,m)$ for which
$$\lvert p_B(n)-m^k\rvert\le d.$$ These results support Sun and Merca et al.'s conjectures, as $p_B(n) \rightarrow p(n)$ when $B \rightarrow +\infty.$ To prove this, we reduce the problem to Siegel's Theorem, which guarantees  the finiteness of integral points on curves with genus $\geq 1$.
\end{abstract}

\section{Introduction and Statement of Results}
The distribution of perfect powers is a classic topic in number theory. A milestone is Catalan's conjecture (1844), proven by Mih\u{a}ilescu in 2002 \cite{Mihailescu2004}, which asserts that the Diophantine equation
\[
x^{a}-y^{b}=1 \qquad (x,y,a,b>1)
\]
only has the single solution
$(x,y,a,b)=(3,2,2,3),$
and so the only consecutive perfect powers are $8$ and $9$. Building on this classical narrative, attention has shifted from perfect powers to their relationship with number theoretic functions. Sun \cite{Sun2018, Sun2021} proposed that partition numbers $p(n)$ where $n>1$, are never perfect-powers. This has led to further developments by Merca et al.  \cite{MercaOnoTsai2025}.

We recall these conjectures. Throughout, $p(n)$ denotes the number of (unrestricted) integer partitions of $n$, where $p(0)=1$. The generating function for this function is given by the Euler product
\begin{equation}\label{GenFunction}
P(q)=\sum_{n=0}^{\infty} p(n)q^n=\prod_{n=1}^{\infty}\frac{1}{1-q^n}.
\end{equation}
An integer is a \emph{perfect power} if it has the form $m^k$ with integers $m\ge 2$ and $k\ge2$.
For a fixed exponent $k\ge2$, we define
\begin{equation}\label{Delta}
\Delta_k(n)\ :=\ \min_{m\ge0}\,|\,p(n)-m^k\,|
\end{equation}
for the distance from $p(n)$ to the nearest $k$th power.  The works of Sun \cite{Sun2018, Sun2021} and Merca et al. \cite{MercaOnoTsai2025} concern the following conjecture.

\begin{conjecture*}[Sun, Merca et al.]
The partition function $p(n)$ repels perfect-powers.
\begin{enumerate}[label=\textnormal{(\roman*)}, leftmargin=1.5em]
  \item \textit{(No perfect-powers)} For every $n\ge2$ and $k\ge2$, $p(n)\neq m^k$ for all integers $m\ge 2$.
  \item \textit{(Perfect-power Repulsion)} For $k\ge2$ and $d\ge0$, at most finitely many $n$ satisfy $\Delta_k(n)\le d$.
\end{enumerate}
\end{conjecture*}

\begin{remark} We note that these questions are similar in flavor to the work of Tengely and Ulas \cite{TU} on equal values of special partition functions that have Diophantine interpretations.
\end{remark}

We prove exact analogs of the conjecture for the sequence of partition functions $p_B(n)$ that converges to $p(n)$. Specifically, for $B\geq 2,$ we consider the truncated Euler products
\begin{equation}\label{pBGenFunction}
P_B(q)=\sum_{n\ge0} p_B(n)\,q^n \;=\; \prod_{m=1}^{B}\frac{1}{1-q^{m}}.
\end{equation}
One sees that $p_B(n)$ is the number of partitions of $n$ whose largest part is at most $B$.
Moreover, we also have that $p_B(n)$ is the number of partitions of $n$ with at most $B$ parts by considering the usual Ferrers board conjugation.
Of course, we have $\lim_{B\rightarrow +\infty}p_B(n)=p(n).$
In analogy with $\Delta_k(n)$, we define
\begin{equation}
\Delta^{(B)}_k(n)\ :=\ \min_{m \geq 0}\bigl|\,p_B(n)-m^k\,\bigr|,
\end{equation}
which is distance between $p_B(n)$ and the nearest $k$th power.
We prove the following theorem that confirms perfect-power repulsion for these functions.

\begin{theorem}\label{thm:main}
If $B\geq 4$ and $k\geq 3$, with $k\nmid (B-1)$, then the following are true.
\begin{enumerate}[label=(\roman*), leftmargin=1.5em]
  \item 
  There are at most finitely many $n$ for which $p_B(n)=m^k$.
  \item 
  If $d\ge0$, then we have
  \[
 \# \{\,n\ge1:\ \Delta^{(B)}_k(n)\le d\,\} < +\infty.
  \]
 \end{enumerate}
\end{theorem}

\begin{example} The partition function $p_2(n)$ does not repel perfect-powers. 
To see this, we
note that
\[
\sum_{n\ge0} p_2(n)\,q^n \;=\; \frac{1}{(1-q)(1-q^2)} 
\;=\; \Big(\sum_{a\ge0} q^{a}\Big)\Big(\sum_{b\ge0} q^{2b}\Big).
\]
Taking the Cauchy product gives
\[
\frac{1}{(1-q)(1-q^2)} 
= \sum_{n\ge0} \Big(\#\{(a,b)\in\mathbb{Z}_{\ge0}^2:\; a+2b=n\}\Big)\,q^n.
\]
For a fixed $n$, writing $n=a+2b$ with $a,b\ge0$ is the same as choosing $b=0,1,\dots,\lfloor n/2\rfloor$, 
and then setting $a=n-2b$. Therefore, we have that
\[
p_2(n)=\lfloor n/2\rfloor+1 \qquad (n\ge0).
\]
In particular, for any integer $m\ge1$ and any $k\ge2$, if we take
$n:=2(m^k-1),$
then 
\[
p_2(n)=\Big\lfloor \frac{2(m^k-1)}{2}\Big\rfloor + 1
= (m^k-1)+1
= m^k.
\]
Thus, $p_2(n)$ takes every perfect $k$th power.
\end{example}

\begin{example} If $B=3$, then we have
$$P_3(q)=\frac{1}{(1-q)(1-q^2)(1-q^3)}=\sum_{n\ge0}p_3(n)q^n,
$$
Along each residue class $n\equiv r\pmod 6,$ we have
\[
p_3(6t+r)=Q_r(t)\quad(t\in\mathbb{Z}_{\ge0}),
\]
where the six quadratic polynomials $Q_r\in\mathbb{Z}[t]$ are
\begin{align*}
Q_0(t)&=3t^2+3t+1, &
Q_1(t)&=3t^2+4t+1, &
Q_2(t)&=3t^2+5t+2,\\
Q_3(t)&=3t^2+6t+3, &
Q_4(t)&=3t^2+7t+4, &
Q_5(t)&=3t^2+8t+5.
\end{align*}

\medskip
For $r\in\{0,1,4,5\},$ there are infinitely many $t\ge0$ such that $Q_r(t)$ is a perfect square.
Consequently, $p_3(n)$ is a square for infinitely many $n$ in each of the residue classes
$r\in\{0,1,4,5\}\pmod6$.
Each diophantine equation $Q_r(t)=m^2$ reduces to a Pell-type equation in $(x,m)$ with
discriminant $12$ after completing the square in $t$.

\smallskip
\noindent
\emph{(Case $r=0$)} From $3t^2+3t+1=m^2$,
\[
(6t+3)^2-12m^2=-3.
\]
This is the negative Pell equation $x^2-12m^2=-3$ (with $x=6t+3$), which has infinitely many
integer solutions. For instance, starting with $(x,m)=(3,1)$ and multiplying by
the fundamental unit $7+2\sqrt{12}$ yields an infinite family. The first few $t$ are
$0,7,104,1455,\ldots$, giving 
$$p_3(6t)\in\{1,13^2,181^2,2521^2,\ldots\}.
$$

\smallskip
\noindent
\emph{(Case $r=1$)} From $3t^2+4t+1=m^2$,
\[
(6t+4)^2-12m^2=4.
\]
The Pell-type equation $x^2-12m^2=4$ has infinitely many solutions; e.g. $t=0,8,120,1680,\ldots$
produce $p_3(6t+1)=1,15^2,209^2,2911^2,\ldots$.

\smallskip
\noindent
\emph{(Case  $r=4$)} From $3t^2+7t+4=m^2$,
\[
(6t+7)^2-12m^2=1,
\]
the (positive) Pell equation with fundamental solution $(x,m)=(7,2)$, hence infinitely many
solutions. We obtain $t=0,15,224,\ldots$ and $p_3(6t+4)=2^2,28^2,390^2,\ldots$.

\smallskip
\noindent
\emph{(Case $r=5$)} From $3t^2+8t+5=m^2$,
\[
(6t+8)^2-12m^2=4,
\]
again yielding infinitely many solutions; e.g. $t=1,31,449,\ldots$ and
$p_3(6t+5)=4^2,56^2,780^2,\ldots$.

\smallskip
In each case the corresponding Pell or Pell-type equation has infinitely many solutions because
the unit group of $\mathbb{Z}[\sqrt{12}]$ is infinite; iterating by the fundamental unit
$7+2\sqrt{12}$ produces an infinite family. Therefore $p_3(n)$ assumes square values infinitely
often in the residue classes $n\equiv 0,1,4,5\pmod6$.
\end{example}

\medskip
To prove Theorem~\ref{thm:main}, we use the fact that the partition functions $p_B(n)$ are quasipolynomial for large $n$, meaning each behaves like fixed polynomials along specific arithmetic progressions. This regularity allows us to reduce the conjectures to curves with genus $\geq 1,$ which then allows us to appeal to Siegel's Theorem.
Summing over these progressions implies Theorem~\ref{thm:main}.

\section*{Acknowledgements}
 \noindent 
The author thanks Mircea Merca, Wei-Lun Tsai, and Maciej Ulas for their comments on an earlier version of this note.
The author thanks the Thomas Jefferson Fund,  the NSF
(DMS-2002265 and DMS-2055118) and the Simons Foundation (SFI-MPS-TSM-00013279) for their generous support. The author is an Honorary Member of the Academy of Romanian Scientists.

\section{Nuts and bolts and the proof of Theorem~\ref{thm:main}}\label{sec:2}

This section gathers the necessary structural inputs for our partition function perfect-power repulsion theorem and then demonstrates Theorem~\ref{thm:main}. Subsection~\ref{sec:2.1} records the quasipolynomial description of $p_B(n)$ on residue classes, its growth, discrete spacing, and includes two Diophantine lemmas that manage the proximity to perfect $k$th powers. Subsection~\ref{sec:2.2} combines these tools to prove Theorem~\ref{thm:main}.

\subsection{Nuts and bolts}\label{sec:2.1}

We start with the quasipolynomial (QP) structure for the counting function $p_B(n)$ and the resulting growth and spacing on arithmetic progressions. Throughout,  we assume that $B\ge4$ and $k\ge3$ are fixed integers, where $k\nmid (B-1)$.

\begin{lemma}[QP structure, degree, and spacing]\label{lem:eqp}
If we let $L:=\lcm(1,2,\dots,B),$ then there are polynomials $Q_0,\dots,Q_{L-1}\in\mathbb{Q}[x]$ of the same degree $B-1$ with positive leading coefficients such that for all $n$ we have
\[
p_B(Ln+r)=Q_r(n)\qquad(0\le r<L).
\]
Moreover, for each $r$, one has
\begin{displaymath}
\begin{split}
&Q_r(n)=\alpha_r\,n^{B-1}+O(n^{B-2}) \qquad \quad(\alpha_r>0),\\
&Q_r(n+1)-Q_r(n)=(B-1)\alpha_r\,n^{B-2}+O(n^{B-3}).
\end{split}
\end{displaymath}
\end{lemma}

\begin{proof}
Consider the  generating function (\ref{pBGenFunction})
\[
  P_B(q)=\sum_{n\ge 0}p_B(n)q^n=\prod_{m=1}^B\frac{1}{1-q^m},
\]
and let $L:=\mathrm{lcm}(1,2,\dots,B)$. All of the poles of $P_B(q)$ lie at the $L$th
roots of unity $\zeta$. The pole at $\zeta=1$ has order $B$, and every
other pole has order at most $B-1$. Therefore, we can write the (finite) partial fraction
decomposition
\[
  P_B(q)=\sum_{\zeta^L=1}\;\sum_{j=1}^{e(\zeta)}
  \frac{R_{\zeta,j}(q)}{(1-\zeta q)^j},
\]
where $R_{\zeta,j}(q)$ are polynomials and $e(\zeta)\le B$, with
$e(1)=B$. Extracting coefficients via
$$
  [q^N](1-\zeta q)^{-j}=\zeta^{\,N}\binom{N+j-1}{j-1},
$$  
we obtain
\[
  p_B(N)=\sum_{\zeta^L=1}\;\sum_{j=1}^{e(\zeta)} \zeta^{\,N}\,P_{\zeta,j}(N),
\]
where each $P_{\zeta,j}\in\mathbb Q[N]$ has degree at most $j-1$. Fixing a
residue class $N\equiv r\pmod L$, then we have that each $\zeta^{\,N}=\zeta^{\,r}$ is
constant, and we conclude that $p_B(Ln+r)$ agrees (for all $n\ge 0$) with
a polynomial $Q_r(n)\in\mathbb Q[n]$ of degree at most $B-1$.

The degree is exactly $B-1$ with positive leading coefficient, because the
only term contributing in top degree is the pole at $\zeta=1$ and $j=B$.
Near $q=1$, we have
\[
  \prod_{m=1}^B\frac{1}{1-q^m}
  \sim \frac{1}{\prod_{m=1}^B m}\cdot(1-q)^{-B}
  \qquad(q\to 1),
\]
so
\(
  [q^N](1-q)^{-B}=\binom{N+B-1}{B-1}
  \sim \dfrac{N^{B-1}}{(B-1)!}
\).
Therefore, the leading coefficient equals
$$
  \alpha_r = 
  \dfrac{L^{B-1}}{B!\,(B-1)!}>0.
$$
Finally, if $Q_r(n)=\alpha_r n^{B-1}+O(n^{B-2})$, then the discrete
difference satisfies
\[
  Q_r(n+1)-Q_r(n)=(B-1)\alpha_r\,n^{B-2}+O(n^{B-3}),
\]
by a one-term binomial expansion. This proves the claims.
\end{proof}

The previous lemma is crucial for the proof of Theorem~\ref{thm:main}. The next little lemma establishes that the polynomials $Q_r$ are not simple shifts of perfect powers.

\begin{lemma}\label{lem:no-shifted-powers}
Let $Q(x)\in\mathbb{Q}[x]$ be a polynomial of degree $d\ge 2$ and fix an integer $k\ge 2$.
Define
\[
T_k(Q)\ :=\ \Bigl\{\,t\in\mathbb{Q}\ :\ \exists\,a\in\mathbb{Q}^\times,\ R(x)\in\mathbb{Q}[x]\text{ with }\deg R\ge 1\text{ and }Q(x)-t \,=\, a\,R(x)^k\,\Bigr\}.
\]
Then $T_k(Q)$ is finite. In fact, we have that
\[
T_k(Q)\ \subseteq\ \bigl\{\,Q(\xi)\ :\ \xi\in\overline{\mathbb{Q}} \ \ and \ \ Q'(\xi)=0\,\bigr\},
\]
and so $|T_k(Q)|\le d-1$.
\end{lemma}

\begin{proof}
Suppose $Q(x)-t=a\,R(x)^k$ with $a\in\mathbb{Q}^\times$, $\deg R\ge 1$, and $k\ge 2$.
Differentiating gives
\[
Q'(x)\ =\ a\,k\,R(x)^{k-1} R'(x).
\]
Hence $Q(x)-t$ and $Q'(x)$ have the common nonconstant factor $R(x)^{k-1}$, so $Q(x)-t$ has a multiple root.
Thus, there is $\xi\in\overline{\mathbb{Q}}$ with $Q(\xi)=t$ and $Q'(\xi)=0$ (i.e.\ $t$ is a critical value of $Q$).
Since $Q'$ has degree $d-1$, there are at most $d-1$ such values.
\end{proof}

\begin{corollary}\label{cor:bounded-t-finite}
For any fixed $X\ge 0$, we have
\[
\left | T_k(Q)\ \cap\ \{\,t\in\mathbb{Q}:\ |t|\le X \,\}\right | < +\infty.
\]
\end{corollary}

To prove Theorem~\ref{thm:main}, we require the following notions about polynomials
$Q\in\mathbb{Z}[x]$. We say $Q$ is \emph{generic (relative to $k\geq 2$)} if $Q(x)\neq a\,R(x)^k$ for every $a\in\mathbb{Q}^{\times}$ and $R\in\mathbb{Q}[x]$. Otherwise, we say that $Q$ is \emph{power-type (relative to $k$)}.
The next lemma is a Diophantine inputs. The first gives finiteness of perfect and near-perfect $k$th powers along any progression whose quasipolynomial piece is generic. The second lemma gives a pointwise lower bound in the power-type case.

\begin{lemma}\label{lem:generic-finiteness}
Let $Q(x)\in\mathbb{Q}[x]$ have degree $d\ge 2$, let $k\ge 3$ be an integer, and fix $D\in\mathbb{Z}_{\ge 0}$. Assume that for every $t\in\mathbb{Q}$ with $|t|\le D$ the polynomial $Q(x)-t$ is not of the form $a\cdot R(x)^k$ with $a\in\mathbb{Q}^\times$ and nonconstant $R\in\mathbb{Q}[x]$. Then we have
\[
\#\{(n,m)\in\mathbb{Z}_{\ge 0}\times\mathbb{Z}:\ |Q(n)-m^k|\le D\}<\infty.
\]
Moreover, except possibly in the pair $(d,k)=(2,2)$, this set is  bounded in terms of $Q,k,D$.
\end{lemma}

\begin{proof}
Write $Q=\frac{1}{M}\,Q_e$ with $M\in\mathbb{Z}_{\ge 1}$ and $Q_e\in\mathbb{Z}[x]$ primitive. Then
\[
|Q(n)-m^k|\le D
\quad\Longleftrightarrow\quad
\big|\,Q_e(n)-M\,m^k\,\big|\le MD.
\]
Hence, it suffices to prove finiteness, for each fixed integer $t$ with $|t|\le MD$, of the Diophantine equation
\begin{equation}\label{eq:mainB}
M\,m^k \;=\; Q_e(n)-t.
\end{equation}
Fix such a $t$. Factor $M$ as $M=u^k\,b_0$ with $u\in\mathbb{Z}_{\ge 1}$ and $b_0\in\mathbb{Z}_{\ge 1}$ $k$-th-power-free. Then every integer solution $(n,m)$ to \eqref{eq:mainB} yields an \emph{integral} solution $(X,Y)=(n,\,u\,m)$ of the superelliptic equation
\begin{equation}\label{eq:curve}
b_0\,Y^k \;=\; Q_e(X)-t.
\end{equation}
Conversely, any integral solution $(X,Y)$ to \eqref{eq:curve} with $u\mid Y$ corresponds to a solution $(n,m)=(X,\,Y/u)$ of \eqref{eq:mainB}. Therefore, the set of solutions to \eqref{eq:mainB} is finite if and only if the set of integral points $(X,Y)\in\mathbb{Z}^2$ on \eqref{eq:curve} is finite.

Set $d=\deg Q_e=\deg Q\ge 2$ and, for each fixed $t\in\mathbb{Z}$, define
\[
r_t \ :=\ \#\{\,x\in \overline{\mathbb{Q}}:\ Q_e(x)=t\,\}
\]
(counted \emph{without} multiplicity). Note that the set of $t$ with $r_t<d$ is finite (the critical values of $Q_e$).

\smallskip
\noindent
{\bf Case $r_t=1$.} If $Q_e(X)-t=c\,(X-a)^d$ with $c\in\mathbb{Z}\setminus\{0\}$ and $d=\deg Q_e\ge 2$, 
then \eqref{eq:curve} becomes
\[
b_0\,Y^k \;=\; c\,(X-a)^d.
\]
Write $u=\gcd(k,d)$ and $k=uk_1$, $d=ud_1$ with $\gcd(k_1,d_1)=1$. 
If $k\mid d$ and $c/b_0$ is a $k$th power, then $Q_e(X)-t$ would be a constant times a $k$th power, 
contrary to the hypothesis. Otherwise, by standard finiteness results for Thue/Thue-Mahler type equations 
(for example, see Chapter 9 of \cite{EG2015}), this Diophantine equation has only finitely many integer solutions.

\smallskip
\noindent{\bf Case $r_t\ge 3$.}
Then the smooth projective model of \eqref{eq:curve} has genus $g\ge 1$ (see, e.g., the standard genus formula for superelliptic curves in \S IV of \cite{Griffiths}). By Siegel's theorem (see Chapter 8 of \cite{Lang}), \eqref{eq:curve} has only finitely many integral points.

\smallskip
\noindent{\bf Case $r_t=2$.}
Write $Q_e(X)-t=c(X-a)^2(X-b)$ with $a\neq b$.
If $k\ge 3$, \eqref{eq:curve} is a Thue/Thue-Mahler-type equation, and standard results 
(for example, see Chapter 9 of \cite{EG2015})  imply finiteness of integral solutions.

\smallskip
\noindent{\bf Case $(d,k)=(2,2)$.}
This is the classical conic case, which may have infinitely many integral points (Pell-type).

When $g\ge 1$, Siegel's theorem on integral points (for example, see Chapter 8 of \cite{Lang})  implies that the set of integral solutions $(X,Y)\in\mathbb{Z}^2$ to \eqref{eq:curve} is finite. Because there are only finitely many $t$ with $|t|\le MD$, taking the union over these $t$ shows that there are only finitely many $(n,m)$ with $|Q(n)-m^k|\le D$.
\end{proof}

\subsection{Proof of Theorem~\ref{thm:main}}\label{sec:2.2}

By Lemma~\ref{lem:eqp}, there is a period $L$ and polynomials $Q_0,\dots,Q_{L-1}\in\mathbb{Q}[X]$ of
degree $B-1$ with positive leading coefficients such that $p_B(Ln+r)=Q_r(n)$ for $0\le r<L$.
It suffices to prove that for each fixed residue class $r$ and each fixed $d\ge 0$,
the set
\[
\mathcal{S}_{r,d}\ :=\ \bigl\{(n,m)\in\mathbb{Z}_{\ge 0}\times\mathbb{Z}:\ |Q_r(n)-m^k|\le d\bigr\}
\]
is finite; then a finite union over $r$ gives (ii), and taking $d=0$ gives (i).

Fix $r$ and $d\ge 0$. For each integer $t$ with $|t|\le d$ consider the Diophantine equation
\begin{equation}\label{eq:level-t}
Q_r(n)-m^k \;=\; t, \qquad n\in\mathbb{Z}_{\ge 0},\ m\in\mathbb{Z}.
\end{equation}
Let
\[
T_r(d)\ :=\ \Bigl\{\,t\in\mathbb{Z}:\ |t|\le d\ \text{ and }\ \exists\,a\in\mathbb{Q}^\times,\ R\in\mathbb{Q}[X]
\text{ nonconstant with }Q_r(X)-t=a\,R(X)^k\,\Bigr\}.
\]
By Lemma~\ref{lem:no-shifted-powers} and Corollary~\ref{cor:bounded-t-finite}, $T_r(d)$ is finite.

\smallskip
\noindent
\emph{Case of Non-exceptional shifts.}
For every $t\in\{-d,-d+1,\dots,d\}\setminus T_r(d)$, the polynomial $Q_r(X)-t$ is \emph{not}
a constant times a $k$th power. Choose $b_0\in\mathbb{Z}_{>0}$ so that $Q_e(X):=b_0 Q_r(X)\in\mathbb{Z}[X]$.
Then any solution of \eqref{eq:level-t} gives an integer solution of
\[
b_0\,m^k\ =\ Q_e(n)-b_0 t.
\]
By Lemma~\ref{lem:generic-finiteness} (applied with the fixed polynomial $Q_e(X)-b_0 t$), there are only finitely many such
integer solutions. Hence $\#\{(n,m): Q_r(n)-m^k=t\}<\infty$ for all non-exceptional $t$ with $|t|\le d$.

\smallskip
\noindent
\emph{Case of exceptional shifts.}
Now fix $t\in T_r(d)$. As above, pass to $Q_e(X)=b_0 Q_r(X)\in\mathbb{Z}[X]$ and consider
\[
b_0\,m^k\ =\ Q_e(n)-b_0 t.
\]
Let $r_t$ denote the number of distinct roots of $Q_e(X)-b_0 t$ in $\overline{\mathbb{Q}}$.

\quad$\bullet$ If $r_t\ge 3$, then the affine curve $b_0Y^k=Q_e(X)-b_0 t$ has genus $\ge 1$,
so by Siegel's theorem (for example, see Chapter 8 of \cite{Lang}) there are only finitely many integer points.

\quad$\bullet$ If $r_t=2$, standard Thue-Mahler finiteness applies (for example, see Chapter 9 of \cite{EG2015})
so there are only finitely many integer solutions.

\quad$\bullet$ If $r_t=1$, then $Q_e(X)-b_0 t=c\,(X-a)^d$ for some $a,c\in\mathbb{Z}$, $c\ne 0$, with
$d=\deg Q_e=B-1\ge 3$. Dividing by $b_0$ gives
\[
Q_r(X)-t \;=\; \frac{c}{b_0}\,(X-a)^d.
\]
Since $k\nmid d=B-1$ by hypothesis, the Diophantine equation $b_0\,m^k=c\,(n-a)^d$ is of Thue/Thue-Mahler type with coprime exponents and therefore has only finitely many integer solutions (for example, see Chapter 9 of \cite{EG2015}). Therefore, the case $r_t=1$ also yields at most finitely many solutions.

\smallskip
Combining the three subcases, for each $t\in T_r(d)$ the equation \eqref{eq:level-t} has only finitely
many integer solutions. Therefore $\mathcal{S}_{r,d}$ is the finite union, over the finitely many
$t\in\{-d,\dots,d\}$, of finite sets. This proves that $\mathcal{S}_{r,d}$ is finite. As noted at the
start, summing over $r$ establishes \textup{(ii)}, and taking $d=0$ gives \textup{(i)}. $\ \square$

\end{document}